\documentclass[12pt,twoside,leqno,final]{amsart}
\usepackage{amsmath}
\usepackage{amssymb}

\theoremstyle{plain}
\newtheorem{thm}{Theorem}[section]
\newtheorem{lem}[thm]{Lemma}
\newtheorem{defn}[thm]{Definition}
\newtheorem{prop}[thm]{Proposition}
\newtheorem{cor}[thm]{Corollary}
\newtheorem{rem}[thm]{Remark}

\newcommand{\C}{{\mathbb C}}
\newcommand{\R}{{\mathbb R}}
\newcommand{\N}{{\mathbb N}}
\newcommand{\D}{{\mathbb D}}
\newcommand{\bD}{{\mathbb T}}
\newcommand{\K}{{\mathcal K}}
\newcommand{\PP}{\mathbb{P}}
\newcommand{\PN}{{\mathcal P}}
\renewcommand{\L}{\mathcal L}
\newcommand{\B}{{\mathcal B}}
\newcommand{\e}{\varepsilon}
\newcommand{\p}{\partial}
\newcommand{\z}{\zeta}

\title{Bilinear forms on  weighted Besov spaces}

\author{Carme Cascante and Joan F\`abrega}
\address{Carme Cascante and Joan F\`abrega: Dept.\ Matem\`atica Aplicada i An\`alisi, Universitat  de Barcelona, Gran Via 585, 08071 Barcelona, Spain}
\email{cascante@ub.edu,\quad joan$_{-}$fabrega@ub.edu}

\date{\today}
\thanks{Partially supported by  DGICYT Grant MTM2011-27932-C02-01  and DURSI Grant 2009SGR 1303.}

\begin{document}

\begin{abstract} 
We compute the norm of some bilinear forms on products of weighted Besov spaces 
in terms of the norm of their symbol in a space of pointwise multipliers 
defined in terms of Carleson measures.
\end{abstract}

\maketitle

\section{Introduction}\label{sec:intro}
The object of this paper is the study of some bilinear forms on products of 
weighted holomorphic Besov spaces, and their  relationship with Hankel operators and weak products.

If $\varphi$ and $\psi$ are measurable functions on $\D$ (on $\bD$ if $t=0$) 
such that $\varphi\overline{\psi}\in L^1(d\nu_t)$, let 
\begin{equation}\label{eqn:pars}
\langle\langle \varphi,\psi\rangle\rangle_t
:=\int_\D \varphi\overline \psi d\nu_t,\quad t>0,\qquad \langle\langle
 \varphi,\psi\rangle\rangle_0
:=\int_\bD \varphi\overline \psi d\sigma.
\end{equation}
Here, for $t>0$, we write $d\nu_t(z):= t (1-|z|^2)^{t-1}d\nu(z)$, where $d\nu$ is 
the normalized Lebesgue measure on the unit disk $\D$, and $d\sigma$ 
denotes the normalized Lebesgue measure on the circle $\bD$.

We also consider the pairings 
\begin{equation}\label{eqn:parw}
\langle h,b\rangle_t:=\lim_{r\to 1^{-}}\langle\langle h(rz), b(rz)\rangle\rangle_t,
\end{equation}
whose domain is the subset of $H\times H$ for which the limit exists.
In particular, if either $b\in H\cap L^1(d\nu_t),\,t>0$, or $b\in H^1,\,t=0$, then we have that for any $h\in H(\overline\D)$, 
$\langle h,b\rangle_t=\langle\langle h,b\rangle\rangle_t$.

In this paper we compute the norm of the bilinear form $\Lambda_b(f,g):=\langle fg,b\rangle_t$  
defined on products of weighted Besov spaces with weights of B\'ekoll\'e type,
in terms of the norm of $b$ in a space of pointwise multipliers related to these Besov spaces.

Let us precise these results. 
Throughout the paper we denote by $H:=H(\D)$ (resp. $H(\overline \D)$) the space of holomorphic functions on the unit disk 
$\D$ (resp. on a neighborhood of $\overline\D$).

If $1<p<\infty$ and $t>0$, the B\'ekoll\'e class $\B_{p,t}$ consists of non-negative functions $\theta\in L^1(d\nu_t)$
such that the measures $d\mu_t:=\theta d\nu_t$  and $ d\mu^\prime_t:=\theta^{-p'/p}d\nu_t$ satisfy the so called $\B_p$ condition
\begin{align*}
\B_{p,t}(\theta):&=\sup_{\substack{z\in \D}}
\left(\frac{\mu_t(T_z)}{\nu_t(T_z)}\right)^{1/p}
\left(\frac{\mu^\prime_t(T_z)}{\nu_t(T_z)}\right)^{1/p'}<\infty,
\end{align*}
where $p'$ is the conjugate exponent of $p$, 
$$
T_z:=\{w\in\D:\,|1-w\overline z/|z||<2(1-|z|^2)\},\,\, z\ne 0,\quad \text{   and }\quad T_0:=\D.
$$

If $1\le p<\infty$, $s\in\R$, $\theta\in \B_{p,t}$ and $d\mu_t=\theta d\nu_t$, then the  Besov space $B^p_s(\mu_t)$ 
consists of holomorphic functions $f$ on $\D$ satisfying 
$$
\|f\|_{B^p_s(\mu_t)}^p:=\int_\D \left|(1+R)^{k_s} f(z)\right|^p\,(1-|z|^2)^{(k_s-s)p}\,d\mu_t(z)<\infty.
$$
Here, $k_s:=\min\{k\in \N:\,k>s\}$ and $R$ denotes the radial derivative.

 As it happens for the unweighted case, if we replace $k_s$ by another non-negative 
 integer $k>s$ we obtain 
equivalent norms  (see for instance \cite[Section 3]{Ca-Or1}). In particular, 
if $s<0$, then we can take $k=0$, and thus we have that
$B^p_{s}(\mu_t)=H\cap L^p(\mu_{t-sp})$.

The classical unweighted Besov space $B^p_s$ corresponds to $B^p_s(\mu_0)$, 
where $d\mu_0(z)=\frac{d\nu(z)}{1-|z|^2}$.  
Observe that this space is already included in the scale of weighted Besov spaces we have considered, simply because 
 $B^p_s(\mu_0)=B^p_{s+t/p}(\nu_t)$ for any  $t>0$. 
In order to recover some well-known results for the unweighted case and the pairing
 $\langle\cdot,\cdot\rangle_0$, we write $\B_{p,0}=\{1\}$.

\begin{defn}
The space $CB^p_s(\mu_t)$ consists of the functions $g\in B^p_s(\mu_t)$ for which 
$$
\|g\|_{CB^p_s(\mu_t)}:=\sup_{0\ne f \in B^p_s(\mu_t)} 
\frac{\|f(1+R)^{k_s}g\|_{B^p_{s-k_s}(\mu_t)}}{\|f\|_{B^p_{s}(\mu_t)}}
$$
is finite.
\end{defn}

The space $CB^p_s(\mu_t)$ can be described either in terms of Carleson masures or in terms 
of pointwise multipliers. Indeed, 

\begin{enumerate}
\item \,$b\in CB^p_s(\mu_t)$ if and only if the measure 
$$
d\mu_b(z):=|(1+R)^{k_s} b(z)|^p\,(1-|z|^2)^{(k_s-s)p}\,d\mu_t(z),
$$ 
is a Carleson measure for $B^p_s(d\mu_t)$,
that is, if and only if   the embedding $B^p_s(\mu_t)\subset L^p(d\mu_b)$ is continuous.

\item $b\in CB^p_s(\mu_t)$ if and only if $(1+R)^{k_s}b\in Mult(B^p_s(\mu_t)\to B^p_{s-k_s}(\mu_t)),$
where $Mult(B^p_s(\mu_t)\to B^p_{s-k_s}(\mu_t))$ denotes the space of pointwise multipliers from 
$B^p_s(\mu_t)$ to $B^p_{s-k_s}(\mu_t)$.
\end{enumerate}

When $t=0$, that is for the unweighted case,  we simply denote the space $CB^p_s(\mu_0)$ by $CB^p_s$.

The spaces $CB^p_s$ appear naturally when dealing with some problems on operators on $B^p_s$. For instance, it is well known that 
$Mult(B^p_s)=H^\infty\cap CB^p_s$. 
In some special cases it is not difficult to give a full description of the space $CB^p_s$.
If $s>1/p$, then  $B^p_s$ is a multiplicative algebra and  $CB^p_s=B^p_s$. 
If $s<0$, then it is easy to check that $CB^p_s$ coincides with the Bloch space $B^\infty_0$.
Different type of characterizations of the spaces $CB^p_s$, for $0\le s\le 1/p$, have been obtained by several authors 
(see for instance \cite{Ste},  \cite{Ar-Ro-Saw}, \cite{Ro-Wu}, \cite{Ar-Ro-Saw-W2}, \cite{Ca-Or1}, \cite{Ca-Or2}  
 and the references therein).

One of the main results of this paper is the following theorem.

\begin{thm} \label{thm:BF}
Let $1<p<\infty$, $0< s<1$, $t\ge 0$ and $\theta \in \B_{p,t}$. 
For $b\in H(\D)$ the following assertions are equivalent:
\begin{enumerate}
	\item \label{item:BF1} $b\in CB^p_{s}(\mu_t)$.

	\item \label{item:BF3} $\displaystyle{\Gamma_1(b)
	:=\sup_{0\ne f,g\in H(\overline \D)}\frac{|\langle\langle |fg|,|(1+R)b|\rangle\rangle_{t+1}|}
	{\|f\|_{B^p_{s}(\mu_t)}\|g\|_{B^{p'}_{-s}(\mu^\prime_t)}}<\infty}$.

	\item \label{item:BF2} $\displaystyle{
\Gamma_2(b):=\sup_{0\ne f,g\in H(\overline \D)}
	\frac{|\langle fg,b\rangle_t|} {\|f\|_{B^p_{s}(\mu_t)}\|g\|_{B^{p'}_{-s}(\mu^\prime_t)}}<\infty}$.
\end{enumerate}

Moreover,  
$
\|b\|_{CB^p_s(\mu_t)}\approx \Gamma_1(b)
\approx \Gamma_2(b).
$
\end{thm}

The symbol $\approx$ means here that each term is bounded by constant times the other term, 
with constants which do  not depend of  the function $b$.

If $b\in L^1(d\nu_t)$, then the small Hankel operator $\mathfrak{h}_b^t$, $t\ge 0$,  is defined  on $ H(\overline{\D})$ by
$$
\mathfrak{h}_b^t(f)(z)
:=\int_\D \,f(w)\overline{b(w)} \frac{d\nu_t(w)}{(1-w\overline z)^{1+t}}, \quad t>0,
\qquad 
\mathfrak{h}^0_b(f)(z):=\int_\bD \, \frac{f(\z)\overline{b(\z)}}{1-\z\overline z}d\sigma(\z).
$$

Notice that, by Fubini's theorem, if $f,g\in H(\overline \D)$, then 
$\left\langle g,\overline{\mathfrak{h}_b^t(f)}\right\rangle_t=\langle fg,b\rangle_t.$
Thus, we have $\Gamma_2(b)=\|\mathfrak{h}_b^t\|_{\L(B^p_s(\mu_t)\to \overline{B^p_s(\mu_t)})}$.

In the above theorem we compute the norm of the bilinear forms on the product $B^p_{s}(\mu_t)\times B^{p'}_{-s}(\mu^\prime_t)$, However,  
 using that the operator $(1+R)^{s'}$ is a bijection from $B^p_s(\mu_t)$ to $B^p_{s-s'}(\mu_t)$, that $\B_{p,t}\subset \B_{p,t+t_0}$, $t_0\ge 0$ and  
\begin{equation}\label{eqn:correind}
B^p_s(\mu_t)=B^p_{s+t_0/p}(\mu_{t+t_0}),
\end{equation} 
we can use Theorem \ref{thm:BF} to compute norms of bilinear 
forms on products $B^p_{s_0}(\mu_{t_0})\times B^{p'}_{s_1}(\mu^\prime_{t_1})$ for some particular choices of the indexes $s_0$, $s_1$, $t_0$ and $t_1$.
 For instance, we have:
 
\begin{cor} \label{cor:BFG}
Let $1<p<\infty$, $t_0,t_1\ge 0$, $\theta\in \B_{p,t_0}$ and $s_0\in \R$. 
For $s_1\in\R$ satisfying $ s_0+s_1<0$ and $0< \frac{s_0}{p'}-\frac{s_1}{p}<1$,  
let $t=t_0-s_0-s_1$.
 
Then we have
$$
\|R^{t-t_1}_{1+t} b\|_{CB^p_{s_0/p'-s_1/p}(\mu_t)}
\approx \sup_{0\ne f,g\in H(\overline \D)}\frac{|\langle fg,b\rangle_{t_1}|}{\|f\|_{B^p_{s_0}(\mu_{t_0})}\|g\|_{B^{p'}_{s_1}(\mu^\prime_{t_0})}},
$$
where $R^{t-t_1}_{1+t}$ is a fractional differential operator of order $t-t_1$ (see \eqref{eqn:Rts}).
\end{cor}

For $s_0,s_1<0$ we prove the following result:

\begin{thm}\label{thm:predualBinfs}
If $1<p<\infty$, $t\ge 0$, $\theta\in \B_{p,t}$ and $s_0,s_1<0$, then 
$$
\| b\|_{B^\infty_{-s_0-s_1}}
\approx \sup_{0\ne f,g\in H(\overline \D)}\frac{|\langle fg,b\rangle_{t}|}{\|f\|_{B^p_{s_0}(\mu_{t})}\|g\|_{B^{p'}_{s_1}(\mu_{t})}}.
$$
\end{thm}

The results in Theorem \ref{thm:BF} for the unweighted case are stated in a different 
 formulation by different authors. For instance, see  \cite{Ro-Wu} and \cite{Wu} for the case 
 $p=2$, and \cite{Bla-Pau} for $p\ne 2$. See also the references therein.
The proof of our results follow some of the ideas used in \cite{Wu}, 
modifying the Hilbert techniques valid only for the case $p=2$ in order to cover 
 the weighted case and $p\ne 2$.
Our approach permit us to compute the norms of the 
  bilinear form   on  $B^p_{s_0}(\mu_t)\times B^{p'}_{s_1}(\mu^\prime_t)$ only when $s_0<0$ or $s_1<0$.  
 It seems more difficult to compute  this norm for the cases $s_0,s_1>0$. 
 Some results for the unweighted case and $p=2$ can be found for 
 instance in \cite{Wu} ( $s_0>s_1$) and in the recent papers 
  \cite{Ar-Ro-Saw-W1} and \cite{Ca-Or2} ($s_0=s_1=1/2$).

 As it happens in the unweighted case (see for instance \cite{Ro-Wu}, \cite{Co-Ve},  \cite{Ar-Ro-Saw-W2}), from the equivalences between \eqref{item:BF1} and 
\eqref{item:BF3} in Theorem \ref{thm:BF}, we obtain the following duality result for weak products.

\begin{thm} \label{thm:predualCBps}
Let $1<p<\infty$,  $t\ge 0$ and $\theta\in \B_{p,t}$. 
If we consider the pairing $\langle\cdot,\cdot\rangle_t$, we then have:

\begin{enumerate}
\item If $0<s<1$, then 
$
(B^p_s(\mu_t)\odot B^{p'}_{-s}(\mu^\prime_t))'\equiv CB^p_s(\mu_t).
$
	\item If $s_0,s_1<0$, then 
$
\left(B^p_{s_0}(\mu_t)\odot B^{p'}_{s_1}(\mu^\prime_t)\right)'\equiv B^\infty_{-s_0-s_1},
$
 and consequently we have 
$
B^p_{s_0}(\mu_t)$ $\odot B^{p'}_{s_1}(\mu^\prime_t)=B^1_{s_0+s_1-t}.
$
\end{enumerate}
\end{thm}

The same arguments used to prove Corollary \ref{cor:BFG} from Theorem \ref{thm:BF}, 
combining the above theorem with \eqref{eqn:correind}, give a description of the dual of 
$B^p_{s_0}(\mu_{t_0})\odot B^{p'}_{s_1}(\mu^\prime_{t_0})$ for $s_0,s_1$ and $t_0$ 
satisfying the conditions in Corollary \ref{cor:BFG}. 
These results cover some well-known results stated in section 5 in \cite{Co-Ve} 
for the unweighted case.

The paper is organized as follows. 
In Section \ref{sec:notpre} we give some  definitions and  
we  state some properties of the class of weights in $\B_{p,t}$ and its corresponding  weighted Besov spaces.
In Section \ref{sec:snegative} we obtain estimates of $\|b\|_{CB^p_s(\mu_t)}$ 
which in particular give the proof of Theorem \ref{thm:predualBinfs}.
Section \ref{sec:proofBF} is devoted to the proof of Theorem \ref{thm:BF} and Corollary \ref{cor:BFG}. 
In Section \ref{sec:duality}, we use our previous results to prove Theorem \ref{thm:predualCBps}.

\section{Notations and preliminaries} \label{sec:notpre}

Throughout this paper, the expression $F\lesssim G$ means that there exists  
a positive constant $C$  independent of the essential variables and such that $F\le CG$. 
If $F\lesssim G$ and $G\lesssim F$ we will write $F\approx G$. 

\subsection{Differential and integral operators}\quad\par

We denote the partial derivatives of first order by 
$\p:=\dfrac{\p}{\p z}$ and $\overline\p:=\dfrac{\p}{\p \overline z}$ respectively . Let $R:=z\p$ be the radial derivative.

For $s, t\in\R$, $t>0$ and $k$ a non-negative integer, we  consider 
the differential operator $R_t^k$ of order $k$ defined by  
$$
R_t^k f :=\left(1+\frac{R}{t+k-1}\right)\cdots\left(1+\frac{R}t\right)f.
$$

If we need to specify the variable of differentiation, then we write $\p_z$, $R_z$ and $R^k_{t,z}$, respectively. 

The operators $R^k_{t}$ satisfy the following formula:
\begin{equation}\label{eqn:Rtk}
R_t^k\,\frac{1}{(1-z\overline w)^t}=\frac{1}{(1-z\overline w)^{t+k}}.
\end{equation}

\begin{defn}
For $N>0$ and $M\ge 0$, we consider the following integral operators:
$$
\PN^{N,M}(\varphi)(z):=\int_\D \varphi(w) \PN^{N,M}(z,w)d\nu(w),\quad\text{where}\quad  \PN^{N,M}(z,w):=N\frac{(1-|w|^2)^{N-1}}{(1-z\overline w)^{1+M}}.
$$

$$
\PP^{N,M}(\varphi)(z):=\int_\D \varphi(w)\PP^{N,M}(z,w)d\nu(w),\quad\text{where}\quad \PP^{N,M}(z,w):=|\PN^{N,M}(z,w)|.
$$

We extend the definition to the case $N=0$ by writing 
$$
\PN^{0,M}(\varphi)(z):=\int_\bD \frac{\varphi(\z)}{(1-z\overline \z)^{1+M}}d\sigma(\z),\quad \PP^{0,M}(\varphi)(z):=\int_\bD \frac{\varphi(\z)}{|1-z\overline \z|^{1+M}}d\sigma(\z).
$$

If $N=M$, then we denote $\PN^{N,N}$ and $\PP^{N,N}$ by $\PN^N$ and  $\PP^N$, respectively.

For $N\ge 0$, we also define 
$$
\K^N(\overline\p \varphi)(z):=\int_\D \overline\p \varphi(w)\K^N(w,z)d\nu(w),
\quad\text{
where}\quad 
\displaystyle{
\K^N(w,z):=\frac{(1-|w|^2)^N}{(1-z\overline w)^N}\frac{1}{w-z}.
}$$
\end{defn}

The weighted Cauchy-Pompeiu representation formula  is given by:

\begin{thm}\label{thm:CP}
Let $N\ge 0$ and $\varphi\in C^1(\overline\D)$. Then 
$
\varphi(z)=\PN^N(\varphi)(z)+\K^N(\overline\p \varphi)(z).
$
\end{thm}

Since 
$R^k_{1+N} f=R^k_{1+N}\PN^{N}(f)=\PN^{N,N+k}(f)$,
it is natural to extend the definition of $R^k_t$ for a noninteger order by considering  
\begin{equation}\label{eqn:Rts}
R^s_{1+N} f:=\PN^{N,N+s}(f),\qquad s, \, N>0.
\end{equation}

Note that by Theorem \ref{thm:CP} we have
$$
\int_\D \PN^{N+s,N}(w,z)\PN^{N,N+s}(u,w)d\nu(w)=\PN^{N}(u,z).
$$
Therefore, for $s>0$ we can define the inverse of $R^{s}_{1+N}$ by  
$
R^{-s}_{1+N} f:=\PN^{N+s,N}(f).
$

Let us recall the following estimate.

\begin{lem} \label{lem:estP}
If $q<2$, $N>0$, $M\ne N-q$ and $z\in\D$, then
$$
\int_{\D}\frac{\PP^{N,M}(w,z)}{|w-z|^q}d\nu(w)\lesssim (1+(1-|z|^2)^{N-M-q}).
$$
\end{lem}

\begin{proof} The case $q=0$ is well known (see for  instance \cite[Lemma 4.2.2 ]{Zhu1}). 
The case $q\ne 0$ can be reduced to the case $q=0$
using the change of variables $w=\varphi_z(u):=\dfrac{z-u}{1-u\overline z}$. Indeed,
we have
\begin{align*}
\int_{\D}\frac{\PP^{N,M}(w,z)}{|w-z|^q}d\nu(w)
=(1-|z|^2)^{N-M-q}\int_\D \frac{(1-|u|^2)^{N-1}}{|1-u\overline z|^{1+2N-M-q}}\frac{d\nu(u)}{|u|^q},
\end{align*}
which ends the proof.
\end{proof}

\subsection{B\'ekoll\'e weights} \label{sec:Bekolle}\quad\par

In this section we recall some properties of the B\'ekoll\'e weights $\B_{p,t}$. We refer to \cite{Be} for  more details. 
Recall that  if $t>0$ and $\theta\in \B_{p,t}$, then $d\mu_t=\theta d\nu_t$ and $d\mu^\prime_t=\theta^{-p'/p}d\nu_t$.

Since, for any $w\in T_z$,  $1-|w|^2\le 4(1-|z|^2)$, we have:

\begin{lem} \label{lem:embedBpt}
If $1<p<\infty$, $0<t_0<t_1$ and $\theta\in \B_{p,t_0}$, then 
$\B_{p,t_1}(\theta)\lesssim \B_{p,t_0}(\theta)$.  
Thus, $\B_{p,t_0}\subset \B_{p,t_1}$. 
\end{lem}

The next result was proved in  \cite[Theorem 1 and Propositions 3, 5]{Be}
\begin{thm} \label{thm:opBt}
Let $1<p<\infty$, $t>0$ and let $\theta$ be a positive locally integrable function $\theta$ on $\D$. Then, the following assertions are equivalent:
\begin{enumerate}
	\item \label{item:opBt1}$\theta\in \B_{p,t}$. 
	\item \label{item:opBt2}The integral operator $\PP^{t}$ is bounded on $L^p(d\mu_t)$.
	\item \label{item:opBt3} The integral operator $\PN^{t}$ is bounded on $L^p(d\mu_t)$.	
\end{enumerate}
\end{thm}

It is well known that any weight in the Muckenhoupt class $A_{p}$ satisfies a doubling condition. 
Similarly to what happens for these classes of weights, any weight in $\B_{p,t}$ satisfies a doubling type condition 
with respect to tents. 

We also have a characterization of weights in $\B_{p,t}$ in terms of the kernels $\PP^{t,M}$, 
which is analogous to the one satisfied for the weights in $A_p$ (see \cite{Pe-W}, \cite{Ca-Fa-Or}).

\begin{prop}\label{prop:BP}
Let $1<p<\infty$, $t>0$ and  $\theta\in \B_{p,t}$. We then have:

\begin{enumerate}
\item \label{item:BP1} The measure $\mu_t$ satisfies the following doubling type measure condition:

if $0<r_1<r_2<1$ and $\z\in\bD$, then 
$$
\frac{\mu_t(T_{r_1\z})}{\mu_t(T_{r_2\z})}\le
 \B_{p,t}(\theta)^p\left(\frac{\nu_t(T_{r_1\z})}{\nu_t(T_{r_2\z})}\right)^p
  \approx \B_{p,t}(\theta)^p\left(\frac{1-r_1}{1-r_2}\right)^{(1+t)p} .
$$

\item \label{item:BP2} 
If   $M>(1+t)(\max\{p,p'\}-1)$, the following equivalence holds:
\begin{align*}
\B_{p,t}(\theta)\lesssim \sup_{z\in\D}(1-|z|^2)^M \left(\PP^{t,t+M}(\theta)(z)\right)^{1/p} \left(\PP^{t,t+M}(\theta^{-p'/p})(z)\right)^{1/p'}
\lesssim \B_{p,t}(\theta)^2.
\end{align*}
\end{enumerate}
\end{prop}

\begin{proof}
Part \eqref{item:BP1} follows easily from H\"older's inequality and the fact that $\theta\in \B_{p,t}$. Indeed,  the embedding 
$T_{r_2\z}\subset T_{r_1\z}$ gives
$$
\nu_t(T_{r_2\z})\le\left(\int_{T_{r_2\z}}  \, d\mu_t\right)^{1/p}\left(\int_{T_{r_1\z}} d\mu^\prime_t\right)^{1/p'}
\le \mu_t(T_{r_2\z})^{1/p}\frac{\B_{p,t}(\theta)\,\,\nu_t(T_{r_1\z})}{(\mu_t(T_{r_1\z}))^{1/p}}.
$$
Since $\nu_t(T_{r\z})\approx (1-r)^{1+t}$, we conclude the proof.

In order to prove \eqref{item:BP2} it is enough to prove  the following estimates, valid for $z\in\D$:
\begin{align}
\label{eqn:estB1} \frac{\mu_t(T_z)}{\nu_t(T_z)}&\lesssim (1-|z|^2)^M\PP^{t,t+M}(\theta)(z)\lesssim 
\B_{p,t}(\theta)^p\frac{\mu_t(T_z)}{\nu_t(T_z)},\\
\label{eqn:estB2} \frac{\mu^\prime_t(T_z)}{\nu_t(T_z)}&\lesssim  (1-|z|^2)^M\PP^{t,t+M}(\theta^{-p'/p})(z)\lesssim 
\B_{p,t}(\theta^{-p'/p})^{p'}\frac{\mu^\prime_t(T_z)}{\nu_t(T_z)}.
\end{align}
Observe that \eqref{eqn:estB2}  follows from \eqref{eqn:estB1} since $\theta\in \B_{p,t}$ 
if and only if  $\theta^{-p'/p}\in \B_{p',t}$.

The estimate on the left hand side of  \eqref{eqn:estB1}  is valid for  any $M>0$ and $t>0$, and follows from 
\begin{align*}
\frac{\mu_t(T_z)}{\nu_t(T_{z})}=\frac{1}{\nu_t(T_{z})}\int_{T_{z}}\theta d\nu_t
&\lesssim  (1-|z|^2)^M\int_{T_{z}} \frac{\theta(w)}{|1-w\bar z|^{1+t+M}}d\nu_t(w)\\
&= (1-|z|^2)^M \PP^{t,t+M}(\theta)(z).
\end{align*}

Let us prove the estimate on the right hand side of \eqref{eqn:estB1}.
If $z=0$ then $T_0=\D$ and thus the result is clear. If $z\ne 0$ then let $\z=z/|z|$ and  $J_z$  the integer part of $-\log_2(1-|z|)$.
Consider the sequence $\{z_k\}\subset \D$ defined by
$$
z_k=(1-2^k(1-|z|))\z\,\,\text{if}\,\, k=0,1,\ldots,J_z,\quad \text{and}\quad z_k=0\,\,\text{ if } k>J_z.
$$

 Observe that $z_0=z$ and that  $1-|z_k|^2\approx |1-w\overline{z}|$ for $w\in T_{z_k}\setminus T_{z_{k-1}}$. 
Therefore,
\begin{align*}
(1-|z|^2)^M \PP^{t,t+M}(\theta)(z)
&=(1-|z|^2)^M\sum_{k=0}^{J_z+1} \int_{T_{z_k}\setminus T_{z_{k-1}}}\frac{\theta(w)\,d\nu_t(w)}{|1-w\overline{z}|^{1+t+M}}\\
&\lesssim \sum_{k=0}^{J_z+1} \frac{(1-|z|^2)^M}{(2^k(1-|z|^2))^{1+t+M}}\mu_t(T_{z_k}).
\end{align*}

By  the doubling property \eqref{item:BP1}, we have 
$$
\mu_t(T_{z_k})\lesssim \B_{p,t}(\theta)^p \dfrac{(1-|z_k|)^{(1+t)p}}{(1-|z|)^{(1+t)p}}\mu_t(T_{z})
\approx \B_{p,t}(\theta)^p 2^{k(1+t)p} \mu_t(T_{z}).
$$

Since  $M>(1+t)(p-1)$ and  $\nu_t(T_{z})\approx (1-|z|^2)^{1+t}$ we obtain
$$
(1-|z|^2)^M \PP^{t,t+M}(\theta)(z)\lesssim \B_{p,t}(\theta)^p \frac{\mu_t(T_z)}{\nu_t(T_z)},
$$
which concludes the proof of the right hand side estimate in \eqref{eqn:estB1}.
\end{proof} 

As a consequence of the above proposition and the estimate $1-|w|^2\le 2|1-z\overline w|$, 
we obtain:
\begin{cor}\label{cor:BP}
If $1<p<\infty$, $t\ge 0$, $N>0$,  $M>(1+t+N)(\max\{p,p'\}-1)$ and $\theta\in B_{p,t}$, then
$$
\sup_{z\in\D}(1-|z|^2)^M \left(\PP^{t+N,t+N+M}(\theta)(z)\right)^{1/p} \left(\PP^{t+N,t+N+M}(\theta^{-p'/p})(z)\right)^{1/p'}\lesssim \B_{p,t}(\theta)^2.
$$
\end{cor}

\subsection{Weighted Besov spaces}\quad\par

In this section we recall some properties of the weighted Besov spaces $B^p_s(\mu_t)$ introduced in Section \ref{sec:intro}. 

The next result is well known for the unweighted case (see for instance \cite[Chapters 2, 6]{Zhu2}). The proof for the weighted Besov spaces can be done following the same arguments used to prove Theorem 3.1 in \cite{Ca-Or1}.

\begin{prop} \label{prop:Bpsk} 
Let $1<p<\infty$, $s\in\R$, $t\ge 0$ and $\theta\in \B_{p,t}$.
If $k>s$ is a nonnegative integer, then  
\begin{align*}
\int_\D | D^k f(z)|^p (1-|z|^2)^{(k-s)p}d\mu_t(z)\quad\text{and}\quad
\sum_{m=0}^k\int_\D |\p^k f(z)|^p (1-|z|^2)^{(k-s)p}d\mu_t(z)
\end{align*}
provide equivalent norms on $B^p_s(\mu_t)$, where $D^k$ is either $(1+R)^{k}$ or $R^k_L$.  
\end{prop}

The next embedding relates weighted and unweighted Besov spaces.

\begin{lem}\label{lem:embed} 
If $1<p<\infty$, $s\in\R$, $t\ge 0$ and $\theta\in \B_{p,t}$, then  
$ B^p_s(\mu_t)\subset B^1_{s-t}$.
\end{lem}

\begin{proof}
Since for any positive integer $k$ we have $B^p_s(\mu_t)=(1+R)^{-k} B^p_{s-k}(\mu_t)$ and
 $B^{1}_{s-t}=(1+R)^{-k}B^{1}_{s-t-k}$, it is sufficient to prove the above embedding for $s<0$.

In this case, H\"older's inequality gives
\begin{align*}
\|f\|_{ B^1_{s-t}}&\le \left( \int_\D |f|^pd\mu_{t-sp}\right)^{1/p}\left( \int_\D d\mu^{\prime}_{t}\right)^{1/p'},
\end{align*}
which proves the result.
\end{proof}

In order to state a duality relation between weighted Besov spaces, we need the next lemma.

\begin{lem} \label{lem:ptspar}
The pairing $\langle \cdot,\cdot \rangle_\delta$ defined in \eqref{eqn:parw} satisfies that for  $f,g\in H(\overline\D)$:
\begin{enumerate}
\item \label{item:ptspar1}
$\displaystyle{
\langle f,g\rangle_\delta 
= \left\langle f,R_{\delta +1}^k g\right\rangle_{\delta+k} 
= \left\langle R_{\delta +1}^k f,g\right\rangle_{\delta+k}. 
}$

\item \label{item:ptspar2}
If $\tau \in\R$ then we have
$
\langle f,g\rangle_\delta =\langle (1+R)^\tau f,(1+R)^{-\tau} g\rangle_\delta. 
$
\end{enumerate}
\end{lem}

\begin{proof} 
Let us prove  \eqref{item:ptspar1} for $k=1$, that is 
$$
\langle f,g\rangle_\delta 
= \left\langle f,\left(1+\frac{R}{\delta+1}\right)g\right\rangle_{\delta+1} 
= \left\langle \left(1+\frac{R}{\delta+1}\right)f,g\right\rangle_{\delta+1}. 
$$
Observe that the second equality can be deduced from the first one by conjugation.

If $\delta=0$, then Stokes' theorem gives
\begin{align*}
\langle f,g\rangle_0 &=\frac{1}{2\pi i}\lim_{r\to 1^{-}}\int_\bD f(r\z)\overline{g(r\z)}\,\overline{\z}d\z
=\lim_{r\to 1^{-}}\int_\D \overline{\p}\left(\overline{z} f(rz)\overline{g(rz)}\right)\,d\nu(z)\\
&=\lim_{r\to 1^{-}}\int_\D  f(rz)\overline{((1+R)g)(rz)}\,d\nu(z)=\langle f,(1+R)g\rangle_{1}.
\end{align*}

The case $\delta>0$ follows from the identity 
$$
\delta(1-|z|^2)^{\delta-1}
=(\delta+1)(1-|z|^2)^\delta-\overline\p\left(\overline z(1-|z|^2)^\delta\right),
$$ 
and  integration by parts.

A simple iteration of these identities gives \eqref{item:ptspar1}.

Assertion \eqref{item:ptspar2} follows from the facts that $(1+R)^\tau z^m=(1+m)^\tau z^m$ and that 
$\langle z^k,z^m\rangle_\delta=0$, $k\ne m$.
\end{proof}

The next result extends the well known  duality  $(B^p_s)' \equiv B^{p'}_{-s}$ for the  case  $t=0$ (see \cite{Lu}).

\begin{prop}\label{prop:dualB}
Let $1<p<\infty$, $t\ge 0$ and $\theta\in \B_{p,t}$. If $s\in\R$, then, the dual of $B^p_{s}(\mu_t)$ with respect to the pairing $\langle \cdot,\cdot\rangle_t$ is the Besov space  $B^{p'}_{-s}(\mu^\prime_t)$.
\end{prop}

\begin{proof}
As in the unweighted case, from the duality 
$(L^p(\mu_t))'\equiv L^{p'}(\mu_t^{\prime})$, 
with respect to the pairing $\langle\langle\cdot,\cdot\rangle\rangle_{t+1}$, Theorem \ref{thm:opBt} 
and the Hahn-Banach theorem, we obtain  
$$
\left(B^p_{-1/p}(\mu_t)\right)'=(H\cap L^p(\mu_t))'\equiv H\cap L^{p'}(\mu_t^{\prime})= B^{p'}_{-1/p'}(\mu^\prime_t),
$$
with respect to the pairing $\langle\langle\cdot,\cdot\rangle\rangle_{t+1}$, and consequently with respect to 
the pairing $\langle\cdot,\cdot\rangle_{t+1}$.

Next, we use the above result and Lemma \ref{lem:ptspar} to prove the general case.

If $g\in B^{p'}_{-s}(\mu_t^\prime)$ and $f\in B^{p}_{s}(\mu_t)$, then 
\begin{align*}
|\langle f,g\rangle_t|=|\langle R^1_t f,g\rangle_{t+1}|
\le \|g\|_{L^{p'}(\mu^\prime_{sp'+t})}\|R^1_t f\|_{L^p(\mu_{(1-s)p+t})}
\approx \|g\|_{B^{p'}_{-s}(\mu_t^\prime)}\|f\|_{B^p_s(\mu_t)}.
\end{align*}
Thus, the map $g\to \langle \cdot,g\rangle_t$ is an injective map from $B^{p'}_{-s}(\mu_t^\prime)$ to   $\left(B^p_s(\mu_t)\right)'$.

Let us prove that this map is surjective. 
If $\Lambda$ is a linear form on $B^p_s(\mu_t)$, then $\Lambda\circ (1+R)^{-s-1/p}$ is also a linear form on $B^p_{-1/p}(\mu_t)$. 
Thus, there exists $g\in B^{p'}_{-1/p'}(\mu^\prime_t)$ such that for any $h\in B^p_{-1/p}(\mu_t)$, 
\begin{align*}
\Lambda\circ (1+R)^{-s-1/p}(h)&=\langle h,g\rangle_{t+1}
=\langle (1+R)^{-s-1/p}h, (1+R)^{s+1/p}g\rangle_{t+1}\\
&=\langle (1+R)^{-s-1/p}h, R^{-1}_{1+t}(1+R)^{s+1/p}g\rangle_{t},
\end{align*}
where in the second identity we have used \eqref{item:ptspar2} in Lemma \ref{lem:ptspar} and in the last one \eqref{item:ptspar1} in the same lemma.

Since for any $f\in B^p_s(\mu_t)$, we have that $h=(1+R)^{s+1/p}(f)\in B^p_{-1/p}(\mu_t)$, 
we deduce that 
$\Lambda (f)=\langle f,G\rangle_t$ with $G:=R^{-1}_{1+t}(1+R)^{s+1/p}g\in B^{p'}_{-s}(\mu^\prime_t)$. 
\end{proof}

\begin{cor}\label{cor:dualB}
Let $1<p<\infty$, $t'>t\ge 0$ and $\theta\in \B_{p,t}$. 
If $s\in\R$, then $\left(B^p_{s}(\mu_t)\right)'=B^{p'}_{-s+t-t'}(\mu^\prime_t)$ 
with respect to the pairing $\langle \cdot,\cdot\rangle_{t'}$.

In particular, if $t=0$, then $(B^p_s)' \equiv B^{p'}_{-s-t'}$, with respect to the pairing $\langle \cdot,\cdot\rangle_{0}$.
\end{cor}

\begin{proof}
By the above proposition, we have 
$$
\left(B^p_{s}(\mu_t) \right)'\equiv \left(B^p_{s+(t'-t)/p}(\mu_{t'})\right)' 
\equiv \left( B^{p'}_{-s-(t'-t)/p}(\mu^\prime_{t'})\right)
=\left( B^{p'}_{-s+t-t'}(\mu^\prime_{t})\right)
$$
which ends the proof.
\end{proof}

\section{Estimates of $\|b\|_{CB^p_s(\mu_t)}$ and proof 
of Theorem \ref{thm:predualBinfs}}\label{sec:snegative}

We introduce a variation in the definition of 
the constants $\Gamma_1(b) $ and $\Gamma_2(b) $ in Theorem \ref{thm:BF}, 
which allow us to cover some general situations.

\begin{defn}
If $1<p<\infty$, $s_0,s_1\in\R$, $t\ge 0$, $\theta\in \B_{p,t}$ and $b\in H$, then 
$$
\Gamma_3(b)=\Gamma(b,p,s_0,s_1,t):=\sup_{0\ne f,g\in H(\overline B)}
\frac{|\langle fg,b\rangle_t|}{\|f\|_{B^p_{s_0}(\mu_t)}\,\|g\|_{B^{p'}_{s_1}(\mu^\prime_t)}}.
$$
\end{defn}

We will start proving the following theorem.

\begin{thm}\label{thm:normG3}
Let $1<p<\infty$, $s_0,s_1\in\R$, $t\ge 0$ and $\theta\in \B_{p,t}$. 
Then $\|b\|_{B^\infty_{-s_0-s_1}}\lesssim \Gamma_3(b).$

If $s_0,s_1<0$, then the converse inequality holds.
\end{thm}

The proof of this result will be a consequence of Lemmas \ref{lem:BinG3} and \ref{lem:G3inB}.

\begin{lem} \label{lem:fznorms}
Let $1<p<\infty$, $s_0,s_1\in\R$, $t\ge 0$ and $\theta\in \B_{p,t}$. Let 
\begin{equation}\label{eqn:fznorms}
\tau>\lambda:=(1+t)(\max\{p,p'\}-1)+\max\{0,-s_0p,-s_1 p'\}.
\end{equation}

For  $z\in\D$, we consider the functions  
$$
f_z(w)= \frac{1}{(1-w\overline z)^{(1+t+\tau)/p}}
\quad\text{and}\quad
g_z(w)= \frac{1}{(1-w\overline z)^{(1+t+\tau)/p'}}.
$$

Then 
$$
\|f_z\|_{B^p_{s_0}(\mu_t)}\,\|g_z\|_{B^{p'}_{s_1}(\mu^\prime_t)}
\lesssim \B_{p,t}(\theta)^2(1-|z|^2)^{-\tau-s_0-s_1}.
$$
\end{lem}

\begin{proof}
If $m>s_0$ is a non-negative integer, then 
\begin{align*}
\|f_z\|_{B^p_{s_0}(\mu_t)}^p\approx 
\int_\D \frac{(1-|w|^2)^{t+(m-s_0)p-1}}{|1-z\overline w|^{1+t+\tau+mp}}\theta(w)d\nu(w).
\end{align*}
Analogously, if $m>s_1$, then 
\begin{align*}
\|g_z\|_{B^{p'}_{s_1}(\mu^{\prime}_t)}^{p'}\approx
\int_\D \frac{(1-|w|^2)^{t+(m-s_1)p'-1}}{|1-z\overline w|^{1+t+\tau+mp'}}\theta^{-p'/p}(w)d\nu(w).
\end{align*}

Therefore, if $N,M$ satisfy $0<N<\min\{(m-s_0)p,(m-s_1)p'\}$ and $(1+t+N)(\max\{p,p'\}-1)<M<\min\{k\tau+s_0p,\tau+s_1p'\}$, then 
the estimate $1-|z|^2\le 2|1-w\overline z|$ and Corollary \ref{cor:BP} give
\begin{align*}
&\|f_z\|_{B^p_{s_0}(\mu_t)}\|g_z\|_{B^{p'}_{s_1}(\mu^{\prime}_t)}\\
&\lesssim (1-|z|^2)^{M-\tau-s_0-s_1}
\left(\PP^{t+N,t+N+M}(\theta)(z)\right)^{1/p}\left(\PP^{t+N,t+N+M}(\theta^{-p'/p})(z)\right)^{1/p'}\\
&\lesssim \B_{p,t}(\theta)^2 (1-|z|^2)^{-\tau-s_0-s_1},
\end{align*}
which ends the proof.
\end{proof}

\begin{lem} \label{lem:BinG3}
Let $1<p<\infty$, $s_0,s_1\in\R$, $t\ge 0$, $\theta\in \B_{p,t}$ and $b\in H$.  Then
$\|b\|_{B^\infty_{-s_0-s_1}}\lesssim \Gamma_3(b)$.
\end{lem}

\begin{proof}
We want to prove that for some positive integer $k$, we have 
$$
\|b\|_{B^\infty_{-s_0-s_1}}\approx \sup_{z\in\D}(1-|z|^2)^{k+s_0+s_1}|R^k_{1+t} b(z)|\lesssim \Gamma_3(b).
$$

By Cauchy formula, we have
\begin{align*}
R^k_{1+t} b(z)&=t\lim_{r\to1^{-}}R^k_{1+t}\int_\D b(rw) 
\frac{(1-|w|^2)^{t-1}}{(1-rz\overline  w)^{1+t}}d\nu(w)\\
&=t\lim_{r\to1^{-}}\int_\D b(rw) 
\frac{(1-|w|^2)^{t-1}}{(1-rz\overline  w)^{1+t+k}}d\nu(w).
\end{align*} 

Assume that $k$ is a positive integer satisfying  \eqref{eqn:fznorms}, and let 
$$
f_z(w)= \frac{1}{(1-w\overline z)^{(1+t+k)/p}}
\quad\text{and}\quad
g_z(w)= \frac{1}{(1-w\overline z)^{(1+t+k)/p'}}.
$$
Since
$
|R^k_{1+t} g(z)|=|\langle f_zg_z,b\rangle_t|, 
$
Lemma \ref{lem:fznorms} gives 
$$
|R^k_{1+t} b(z)|\le \Gamma_3(b)
\|f_z\|_{B^p_{s_0}(\mu_t)}\,\|g_z\|_{B^{p'}_{s_1}(\mu^\prime_t)}
\lesssim \Gamma_3(b)(1-|z|^2)^{-k-s_0-s_1},
$$
which concludes the proof.
\end{proof}

\begin{cor} \label{cor:Bloch} 
Let $1<p<\infty$ and $0<s<1$. 
If $b$ satisfies condition \eqref{item:BF2} in Theorem \ref{thm:BF}, that is $\Gamma_3(b,p,s,-s,t)<\infty$, then 
$b\in B^p_{s}(\mu_t)\cap B^\infty_0$.
\end{cor}

\begin{proof}
The above lemma gives $b\in B^\infty_0$. 
The fact that $b\in B^p_s(\mu_t)$ follows from the estimate 
$
|\langle g,b\rangle_t|\le C_b \|1\|_{B^p_s(\mu_t)}\|g\|_{B^{p'}_{-s}(\mu^\prime_t)}
$
and the duality result in Proposition \ref{prop:dualB}.
\end{proof}

\begin{lem} \label{lem:G3inB}
If $1<p<\infty$ and $s_0,s_1<0$, then $\Gamma_3(b)\lesssim \|b\|_{B^\infty_{-s_0-s_1}}$.
\end{lem}

\begin{proof}
Let $k$ be a positive integer such that $k>-s_0-s_1$. Then
\begin{align*}
|\langle fg,b\rangle_t|&=|\langle fg,R^k_{1+t}b\rangle_{t+k}|
\lesssim \|b\|_{B^\infty_{-s_0-s_1}}\|fg\|_{L^1(d\nu_{t-s_0-s_1})}\\
&\le \|b\|_{B^\infty_{-s_0-s_1}}\|f\|_{L^p(\theta d\nu_{t-s_0p})}\|g\|_{L^{p'}(\theta^{-p'/p} d\nu_{t-s_1p'})}\\
&\approx \|b\|_{B^\infty_{-s_0-s_1}}\|f\|_{B^p_{s_0}(\mu_t)}\,\|g\|_{B^{p'}_{s_1}(\mu^\prime_t)}, 
\end{align*} 
which ends the proof.
\end{proof}

\begin{proof}[Proof of Theorem \ref{thm:predualBinfs}]
The proof is an immediate consequence of Lemmas \ref{lem:BinG3} and \ref{lem:G3inB}.
\end{proof}

\begin{thm}\label{thm:normCB}
Let $1<p<\infty$, $s<1$, $t\ge 0$ and $\theta\in \B_{p,t}$. Then, $CB^p_s(\mu_t)\subset B^p_s(\mu_t)\cap B^\infty_0$. 
If $s<0$, then $CB^p_s(\mu_t)=B^\infty_0$.
\end{thm}

\begin{proof}
The first inclusion  follows from the same arguments used to prove Lemma \ref{lem:BinG3}. 
For a non-negative integer $k>s$ which we precise later, we have 
\begin{align*}
&|R_{1+t+(1-s)p}^{k}(I+R) b(z)|=|R_{1+t+(1-s)p}^{k}\PN^{t+(1-s)p}((I+R)b)(z)|\\
&=|\PN^{t+(1-s)p,t+(1-s)p+k}((I+R) b)(z)|\\
&\le \left(\int_\D \frac{(1-|w|^2)^{(1-s)p+t-1}|(I+R)b(w)|^p}
{|1-w\overline z|^{1+t+(1-s)p+k}}\theta(w)d\nu(w)\right)^{1/p}\,\\
&\qquad\qquad \cdot\left( \int_\D\frac{(1-|w|^2)^{(1-s)p+t-1}}{|1-w\overline z|^{1+t+(1-s)p+k}}\theta^{-p'/p}(w)d\nu(w)\right)^{1/p'}\\
&\lesssim \|b\|_{CB^p_s(\mu_t)}\left(\int_\D \frac{(1-|w|^2)^{(1-s)p+t-1}}
{|1-w\overline z|^{1+t+(1-s)p+k+p}}\theta(w)d\nu(w)\right)^{1/p}\,\\
&\qquad\qquad \cdot\left( \int_\D\frac{(1-|w|^2)^{(1-s)p+t-1}}{|1-w\overline z|^{1+t+(1-s)p+k}}\theta^{-p'/p}(w)d\nu(w)\right)^{1/p'}\\
&\le \|b\|_{CB^p_s(\mu_t)}
(1-|z|^2)^{-1}\left(\PP^{t+(1-s)p,t+(1-s)p+k}(\theta)(z)\right)^{1/p}\,\\
&\qquad\qquad\cdot\left(\PP^{t+(1-s)p,t+(1-s)p+k}(\theta^{-p'/p})(z)\right)^{1/p'}.
\end{align*}
If $k>(1+t+(1-s)p)(\max\{p,p'\}-1)$, then Corollary \ref{cor:BP}  with $N=(1-s)p$ and $M=k$, gives 
$
|R_{1+t+(1-s)p}^{k}(I+R) b(z)|\lesssim \|b\|_{CB^p_s(\mu_t)}(1-|z|^2)^{-1-k} 
$
 which proves that $b\in B^\infty_0$.

Next, if $s<0$, then we have $k_s=1$ and the inequality $  \|b\|_{CB^p_s(\mu_t)}\lesssim \|b\|_{B^\infty_0}$
follows from 
$$
\int_\D |f(z)|^p\,|(1+R) b(z)|^p\,(1-|z|^2)^{(1-s)p}\,d\mu_t(z)\lesssim \|b\|_{B^\infty_0}^p\|f\|_{B^p_s(\mu_t)}^p,
$$
which concludes the proof.
\end{proof}

\begin{rem}\label{rem:CBindepk}
Observe that if $0<s<1$, $0<\e<1-s$ and $\|g\|_{B^\infty_{s+\e-1}}<\infty$, then 
$$
\|g f\|_{B^p_{s-1}(\mu_t)}\lesssim \|g\|_{B^\infty_{s+\e-1}} \|f\|_{B^{p}_{-\e}(\mu_t)}\lesssim \|g\|_{B^\infty_{s+\e-1}} \|f\|_{B^{p}_{s}(\mu_t)}.
$$
Therefore, $g\in Mult(B^{p}_{s}(\mu_t)\to B^{p}_{s-1}(\mu_t))$. 
In particular, 
$$
B^\infty_0\subset B^\infty_{s+\e-1}\subset Mult(B^{p}_{s}(\mu_t)\to B^{p}_{s-1}(\mu_t)).
$$

This gives  that 
$g\in CB^{p}_{s}(\mu_t)$ if and only if for some (any) $l>0$, $(l+R)g\in Mult(B^{p}_{s}(\mu_t)\to B^{p}_{s-1}(\mu_t))$. 
\end{rem}

\section{Proof of Theorem \ref{thm:BF} and Corollary \ref{cor:BFG}} \label{sec:proofBF}

\subsection{Proof of \eqref{item:BF1} $\Longrightarrow$ \eqref{item:BF3} $\Longrightarrow$ \eqref{item:BF2} in Theorem \ref{thm:BF}} \quad\par

The fact that \eqref{item:BF1} $\Longrightarrow$ \eqref{item:BF3} is a consequence of H\"older's inequality.
Indeed, since $0<s<1$, we have 
\begin{align*}
\langle \langle|fg|,|(1+R)b|\rangle\rangle_{t+1}
&\le \|g\|_{L^{p'}(\theta^{-p'/p}d\nu_{sp'+t})}\|f(1+R)b\|_{L^p(\theta d\nu_{(1-s)p+t})}\\
&\le \|g\|_{B^{p'}_{-s}(\mu^\prime_t)}\,\|f\|_{B^{p}_{s}(\mu_t)}\,\|b\|_{CB^{p}_{s}(\mu_t)}.
\end{align*}

Clearly \eqref{item:BF3} $\Longrightarrow$ \eqref{item:BF2} is a consequence of Lemma \ref{lem:ptspar} (i). 
Indeed, if $|\langle\langle |fg|,|(1+R)b|\rangle\rangle_{t+1}|<\infty$ for any $f,g\in  H(\overline\D)$, 
then  by Corollary \ref{cor:Bloch} (see also Remark \ref{rem:CBindepk}) we have
$|\langle\langle |fg|,|R^1_{t+1}b|\rangle\rangle_{t+1}|<\infty$.
Thus 
$$
|\langle fg,b\rangle_t|=|\langle fg,R^1_{t+1}b \rangle_{t+1}|\le |\langle\langle |fg|,|R^1_{t+1}b|\rangle\rangle_{t+1}|.
$$
which concludes the proof.

Observe that if $b\in CB^p_s(\mu_t)$, the above estimates give 
\begin{equation}\label{eqn:2implies3}
|\langle fg,b\rangle_t|\le \|b\|_{CB^{p}_{s}(\mu_t)} \|f\|_{B^{p}_{s}(\mu_t)} \|g\|_{B^{p'}_{-s}(\mu^\prime_t)}.
\end{equation}
Thus we have
$
\Gamma_2(b)\le \Gamma_1(b) \le \|b\|_{CB^{p}_{s}(\mu_t)}.
$

\subsection{Proof of \eqref{item:BF2} $\Longrightarrow$ \eqref{item:BF1}  in Theorem \ref{thm:BF} for the unweighted case $t=0$}
\quad\par
In the next proposition we use Corollary \ref{cor:Bloch} and the weighted Cauchy-Pompeiu's 
 formula, to give an easy proof of \eqref{item:BF2} $\Longrightarrow$ \eqref{item:BF1} 
 in Theorem \ref{thm:BF} for the unweighted case $t=0$. This last  
case  has been proved using different methods in \cite{Ro-Wu} for $p=2$ and in 
\cite{Bla-Pau} for any $p>1$. Our approach follows the techniques in \cite{Wu}.

\begin{prop} 
Let $1<p<\infty$ and  $0<s<1$. Assume that $b\in H$ satisfies 
$
|\langle fg,b\rangle_0|\le C_b \|f\|_{B^p_s}\|g\|_{B^{p'}_{-s}}$ for any $f,g\in H(\overline\D)$.  
Then $b\in CB^p_s$.
\end{prop}

\begin{proof}
By Lemma \ref{lem:embed} we have $b\in B^{p}_{s}\subset B^1_0$. 
Therefore, for $f\in H(\overline\D)$, the  weighted Cauchy-Pompeiu's representation formula in Theorem \ref{thm:CP}  gives 
\begin{equation}\label{eqn:cp}
(1+R)b(z)\overline {f}=\PN^{1}((1+R)b\overline {f})+\K^{1}((1+R)b\overline {\p f}).
\end{equation}
In order to prove this proposition it is enough to show that the $L^p(d\nu_{(1-s)p})$-norms of
the two terms in the right hand side in \eqref{eqn:cp} are bounded by  a constant times $\|f\|_{B^p_s}$.

The first term  $h=\PN^{1}((1+R) b \overline{f})$ is a holomorphic function  on $\D$. 
Thus, by Corollary \ref{cor:dualB}, it suffices to prove that  
$
|\langle h, g\rangle_{1}|\le C \|f\|_{B^p_s}\|g\|_{B^{p'}_{-s}}
$
for any $g\in H(\overline\D)$.

By Lemma \ref{lem:ptspar}, this follows from $\langle h, g\rangle_{1}=\langle (1+R)b, fg\rangle_1=\langle b, fg\rangle_0$ 
and the hypotheses.

In  order to  estimate  the $L^p( d\nu_{(1-s)p})$-norm of 
$\K^{1}((1+R)b\overline {\p f})$, note that by Corollary \ref{cor:Bloch} we have $b\in B^\infty_0$. This fact, 
H\"older's inequality and the estimates of Lemma \ref{lem:estP}, 
with $\e>0$ small enough to be chosen later on, we have
\begin{align*}
&|\K^{1}((1+R)b\overline {\p f})(z)|^p
\le \|b\|_{B^\infty_0}^p \left(\int_{\D}\frac{|\p f(w)|}{|1-z\overline w||w-z|}d\nu(w)\right)^p\\
&\le \|b\|_{B^\infty_0}^p \int_{\D}\frac{|\p f(w)|^p(1-|w|^2)^{(1-\e)p-1}}
{|1-z\overline w|^{(1-2\e)p}|w-z|}d\nu(w)
\left(\int_{\D}\frac{(1-|w|^2)^{\e p'-1}}
{|1-z\overline w|^{2\e p'}|w-z|}d\nu(w) \right)^{p/p'}\\
&\lesssim \|b\|_{B^\infty_0}\int_{\D}\frac{|\p f(w)|^p(1-|w|^2)^{(1-\e)p-1}}
{|1-z\overline w|^{(1-2\e)p}|w-z|}d\nu(w) (1-|z|^2)^{-\e p}.
\end{align*}

Therefore, if $0<\e<\min\{s,1-s\}$, then the above estimate, Fubini's theorem and Lemma \ref{lem:estP} give
$$
\|\K^{1}((1+R)b\overline {\p f})\|_{L^p(d\nu_{(1-s)p})}\lesssim \|b\|_{B^\infty_0} 
\|\p f\|_{L^p( d\nu_{(1-s)p})}\lesssim \|b\|_{B^\infty_0}\|f\|_{B^p_s},
$$
which ends the proof.
\end{proof}

\subsection{Proof of \eqref{item:BF2} $\Longrightarrow$ \eqref{item:BF1}  in Theorem \ref{thm:BF} for the general case}\quad\par

Observe that if we   use the same arguments of the above section to prove the unweighted case, 
then in the estimate of $\K^{t+1}((1+R)b\,\overline{\p f})$ we will end up with integrals of the type 
$$
\int_\D \frac{(1-|w|^2)^{N-1}}{|1-z\overline w|^{1+M}|w-z|}\,\theta(w) d\nu(w),
$$
which are difficult to estimate because we do not have precise information on $\theta$ near  the diagonal $z=w$.
One method to avoid this difficulty is based in the use of the following modification of the Cauchy-Pompeiu's formula, 
which on one hand avoid the singularity on the diagonal and in other hand increases the power of $(1-|w|^2)$.

\begin{lem}\label{lem:KN}  
Let $t> 0$, $b\in B^\infty_0$ and $f\in H(\overline\D)$. 
For any integer $m\ge 2$, we have 
\begin{align*}
\K^{t+1}((1+R)b\overline {\p f})
&=\K^{t+m}_0((1+R)b\,\overline {\p^2 f})
+\K^{t+m-1}_1((1+R)b\,\overline {\p f})\\
&\quad +\sum_{j=1}^{m-1} Q^{t+j}((1+R)b\overline {R f}),
\end{align*}
where  
\begin{align*}
\K^{t+m}_0((1+R)b\overline {\p^2 f})(z)
&:=-\int_\D \frac{((1+R)b\overline{\p^2 f})(w)}{(1-z\overline w)^{t+m}}\frac{\overline{w-z}}{w-z}d\nu_{t+m+1}(w),\\
\K^{t+m-1}_1((1+R)b\overline {\p f})(z)
&:=(t+m)\int_\D \frac{((1+R)b\overline{\p f})(w)\,(\overline{w-z})}{(1-z\overline w)^{t+m+1}}d\nu_{t+m}(w),\\
Q^{t+j}((1+R)b\overline {R f})(z)&:=
\int_\D ((1+R)b \overline{R f})(w)\frac{d\nu_{t+j+1}(w)}{(1-z\overline w)^{t+j+1}}.
\end{align*} 
\end{lem}

\begin{proof}
Recall that 
$$
\K^{t}(w,z)=\frac{(1-|w|^2)^{t}}{(1-z\overline w)^{t}}\frac{1}{w-z}.
$$
Since $1=\frac{1-|w|^2}{1-z\overline w}+\frac{\overline w(w-z)}{1-z\overline w}$, we have
$$
\K^{t+1}((1+R)b\overline {\p f})(z)=\K^{t+2}((1+R)b\overline {\p f})(z)+Q^{t+1}((1+R)b\overline {R f})(z).
$$

Iterating this formula, we obtain 
$$
\K^{t+1}((1+R)b\overline {\p f})(z)=\K^{t+m}((1+R)b\overline {\p f})(z)
+\sum_{j=1}^{m-1} Q^{t+j}((1+R)b\overline {R f})(z).
$$

An easy computation shows that
\begin{align*}
&\K^{t+m}(w,z)=\frac{(1-|w|^2)^{t+m}}{(1-z\overline w)^{t+m}}\frac{1}{w-z}\\
&=\overline{\p_w}\left(\frac{(1-|w|^2)^{t+m}}{(1-z\overline w)^{t+m}}\frac{\overline{w-z}}{w-z}\right)
+(t+m)\frac{(1-|w|^2)^{t+m-1}(\overline{w-z})}{(1-z\overline w)^{t+m+1}}.
\end{align*}

Fixed $z\in\D$ and $0<\e< 1-|z|$, let $\Omega_{z,\e}:=\D\setminus \{w\in\D:|w-z|<\e\}$. If we apply  Stokes' theorem to the region $\Omega_{z,\e}$ and let $\e\to 0$, we obtain 
\begin{align*}
\K^{t+m}&((1+R)b\overline{\p f})(z)\\
&=-\int_\D ((1+R)b\overline{\p^2 f})(w)\frac{(1-|w|^2)^{t+m}}{(1-z\overline w)^{t+m}}\frac{\overline{w-z}}{w-z}d\nu(w)\\
&\quad +(t+m)\int_\D ((1+R)b\overline{\p f})(w)\frac{(1-|w|^2)^{t+m-1}(\overline{w-z})}
{(1-z\overline w)^{t+m+1}}d\nu(w),
\end{align*}
which concludes the proof.
\end{proof}

\begin{prop} \label{prop:KN}
Let $1< p<\infty$, $0<s<1$, $t> 0$ , $b\in B^\infty_0$, $f\in H(\overline\D)$ and 
$$
\varphi_f(w):=|\p^2 f(w)|(1-|w|^2)^{2-s}+|\p f(w)|(1-|w|^2)^{1-s}.
$$

 Then we have 
\begin{equation}\label{eqn:KN1}
|\K^{t+1}((1+R)b\overline {\p f})(z)| 
\lesssim \|b\|_{B^\infty_0}\,\PP^{t+s,t+1}\left( \varphi_f \right)(z).
\end{equation}

Therefore, if $\theta\in \B_{p,t}$, then 
\begin{equation}\label{eqn:KN2}
\|(1-|z|^2)^{1-s}\K^{t+1}((1+R)b\overline {\p f})(z)\|_{L^p(\mu_t)}\lesssim \|b\|_{B^\infty_0}\,\|f\|_{B^p_s(\mu_t)}.
\end{equation}
\end{prop}

\begin{proof}
The pointwise estimate \eqref{eqn:KN1} follows from Lemma \ref{lem:KN}. Since $1-|w|^2\le 2|1-z\bar w|$ and $|z-w|\le |1-z\bar w|$,
then for $m\ge 3$, we have
\begin{align*}
|\K^{t+1}&((1+R)b\overline {\p f})(z)| \\
&\lesssim \|b\|_{B^\infty_0}\,\left(\PP^{t+m-2+s,t+m-1}\left( \varphi_f \right)(z) 
+\sum_{j=1}^{m-1} \PP^{t+j+s-1,t+j}\left( \varphi_f \right)(z) \right)\\
&\lesssim \|b\|_{B^\infty_0}\,\PP^{t+s,t+1}\left( \varphi_f \right)(z) .
\end{align*}

In order to prove the $L^p(\mu_t)$-norm estimate \eqref{eqn:KN2}, from \eqref{eqn:KN1} we have 
$$
(1-|z|^2)^{1-s}|\K^{t+1}((1+R)b\overline {\p f})(z)|\lesssim  \|b\|_{B^\infty_0}\,\PP^{t+s}\left( \varphi_f \right)(z) 
$$
and thus
$
\|\PP^{t+s}(\varphi_f)\|_{L^p(\mu_t)}\lesssim \|\varphi_f\|_{L^p(\mu_t)}
\lesssim \|f\|_{B^p_s}(\mu_t),
$
which is a consequence of Theorem \ref{thm:opBt} and Proposition \ref{prop:Bpsk}.
\end{proof}

Now we can prove  \eqref{item:BF2} $\Longrightarrow$ \eqref{item:BF1} in Theorem \ref{thm:BF}.

\begin{prop}
If $b$ satisfies condition \eqref{item:BF2} in Theorem \ref{thm:BF}, then $b\in CB^p_{s}(\mu_t)$.
\end{prop}

\begin{proof}
We want to prove that 
$$
\int_\D |f(z)|^p|R^1_{t+1} b(z)|^p (1-|z|^2)^{(1-s)p}d\mu_t(z)\lesssim C_b\|f\|_{B^p_s(\mu_t)}^p.
$$
To do so, by  the Cauchy-Pompeiu's formula in Theorem \ref{thm:CP}, 
\begin{equation}\label{eqn:wcp}
R^1_{t+1}b(z)\overline {f}=\PN^{t+1}(R^1_{t+1} b\overline {f})+\K^{t+1}(R^1_{t+1} b\overline {\p f}),
\end{equation}
we will show that the two terms in the right hand side in  \eqref{eqn:wcp} are both in
 $L^p(\theta d\nu_{(1-s)p+t})$ and that these norms are bounded up to a constant by 
  $\|f\|_{B^p_s(\mu_t)}$.

Since $h=\PN^{t+1}(\overline{f}\,R^1_{t+1} b)$ is a holomorphic function on $\D$, the norm estimate of $h$ is similar to the one for the unweighted case. Indeed, for $g\in H(\overline\D)$ 
Lemma  \ref{lem:ptspar} gives
$$
|\langle h, g\rangle_{t+1}|=|\langle R^1_{t+1}b, fg\rangle_{t+1}|
=|\langle b, fg\rangle_{t}|\le \Gamma_2(b)\|f\|_{B^p_s(\mu_t)}\|g\|_{B^{p'}_{-s}(\mu^\prime_t)}
$$
which, by  Corollary \ref{cor:dualB}, proves that 
$
\|h\|_{L^p(\theta d\nu_{(1-s)p+t})}\le \Gamma_2(b)\|f\|_{B^p_s(\mu_t)}.
$

Using the $L^p(\theta d\nu_{(1-s)p+t})$-norm estimate of  $\K^{t+1}(R^1_tb\overline{\p f})$ given in Proposition \ref{prop:KN}, we conclude the proof.
\end{proof}

\subsection{Proof of Corollary \ref{cor:BFG}} \label{sec:proof BFG}

Using $ B^p_s(\mu_\delta)=B^p_{s+\tau/p}(\mu_{\delta+\tau})$, for $\tau>0$, 
we will deduce the result from Theorem \ref{thm:BF}.

Let $1<p<\infty$, $s_0,s_1\in\R$, $t_0\ge 0$ and $\theta\in\B_{p,t_0}$. Then, for  $t=t_0-s_0-s_1> t_0$ we have 
\begin{align*}
B^p_{s_0}(\mu_{t_0})&=B^p_{s_0+(-s_0-s_1)/p}(\mu_t)=B^p_{s_0/p'-s_1/p}(\mu_t),\quad\text{ and }\\
B^{p'}_{s_1}(\mu^\prime_{t_0})&=B^{p'}_{s_1/p-s_0/p'}(\mu^\prime_t).
\end{align*}

Moreover, since $\langle f g, b\rangle_{t_1}=\langle \PN^t(f g), b\rangle_{t_1}=\langle fg, \PN^{t_1,t}b\rangle_{t}$,
we have 
$$
\frac{|\langle fg, b\rangle_{t_1}|}
{\|f\|_{B^p_{s_0}(\mu_{t_0})}\,\|g\|_{B^{p'}_{s_1}(\mu^\prime_{t_0})}}=
\frac{|\langle fg, \PN^{t_1,t}b\rangle_{t}|}
{\|f\|_{B^p_{s_0/p'-s_1/p}(\mu_t)}\,\|g\|_{B^{p'}_{s_1/p-s_0/p'}(\mu^\prime_t)}}.
$$

Thus, Theorem \ref{thm:BF}, with $0<s:=s_0/p'-s_1/p<1$, gives
$$
\|\PN^{t_1,t} b\|_{CB^p_{s_0/p-s_1/p'}(\mu_t)}
\approx \sup_{0\ne f,g\in H(\overline \D)}\frac{|\langle fg,b\rangle_{t_1}|}{\|f\|_{B^p_{s_0}(\mu_{t_0})}\|g\|_{B^{p'}_{s_1}(\mu_{t_0})}},
$$
which concludes the proof.

\section{Proof of Theorem \ref{thm:predualCBps}}.\label{sec:duality}
\quad\par 

We will determine the predual of $CB^p_s(\mu_t)$ generalizing some results  
for the unweighted case (see for instance \cite{Ro-Wu}, \cite{Ar-Ro-Saw-W2}, 
\cite{Co-Ve} and the references therein).

\subsection{Weak products and the  predual of $CB^p_s(\mu_t)$}

\begin{defn} Given two Banach spaces $X$ and $Y$ of holomorphic functions on $\D$, let  
 $X\odot Y$ be the completion of finite sums $h=\sum_{j=1}^M f_j g_j$, 
 $f_j\in X$, $g_j\in Y$, using the norm
 $$
 \|h\|_{X\odot Y}:=\inf\left\{\sum_{k=1}^N \|\tilde f_k\|_X\|\tilde g_k\|_Y:\,
 \sum_{k=1}^N \tilde f_k\,\tilde g_k=h\right\}.
 $$
\end{defn}

The following well-known proposition, whose proof follows from the own definitions,
 will be used to prove our duality results.

\begin{prop}\label{prop:dualwproduct}
The norm of a linear form $\Lambda$ on $X\odot Y$ coincides with the norm of 
the bilinear form on $X\times Y$ on defined by $\tilde\Lambda(f,g)=\Lambda(fg)$.
\end{prop}

\subsection{Proof of Theorem \ref{thm:predualCBps}}
\begin{proof}
The embedding 
$
i:B^{p'}_{-s}(\mu^\prime_t)\to B^p_s(\mu_t)\odot B^{p'}_{-s}(\mu^\prime_t),
$
shows that any linear form $\Lambda \in (B^p_s(\mu_t)\odot B^{p'}_{-s}(\mu^\prime_t))'$ 
produces a linear form $\Lambda_i=\Lambda\circ i$ \, on $B^{p'}_{-s}(\mu^\prime_t)$, 
which by Proposition \ref{prop:dualB} can be expresed as  
$\Lambda_i(f)=\langle f,b\rangle_t,$ for some $\,b\in B^p_s(\mu_t).$

Consequently, $\Lambda(h)=\langle h,b\rangle_t$ for $h\in H(\overline{\D})$. 
Since $H(\overline{\D})$ is dense in both spaces $B^p_s(\mu_t)$ 
and  $B^{p'}_{-s}(\mu^\prime_t)$, 
then it is also dense in $B^p_s(\mu_t)\odot B^{p'}_{-s}(\mu^\prime_t)$, 
and thus the norm of $\Lambda$ 
coincides with the norm of the bilinear form $(f,g)\to \langle fg,b\rangle_t$ 
on $B^p_s(\mu_t)\times  B^{p'}_{-s}(\mu^\prime_t)$.
 Therefore, the equivalence between \eqref{item:BF1} and \eqref{item:BF2} 
 in Theorem \ref{thm:BF} concludes the proof.

The same arguments used in the first part show that the norm of a linear 
form $\Lambda$ on $B^p_{s_0}(\mu_t)\odot B^{p'}_{s_1}(\mu^\prime_t)$ 
is equivalent to the norm of the bilinear form 
$(f,g)\to \langle fg,b\rangle_t$, where $b\in B^p_{-s_1}(\mu_t)$.
By Theorem \ref{thm:normG3} this norm is equivalent to $\|b\|_{B^\infty_{-s_0-s_1}}$ 
which proves the first statement.

The second statement follows from the computation by duality of the norms  
$\|h\|_{B^1_{s_0+s_1-t}}$ and  $\|h\|_{B^p_{s_0}(\mu_t)\odot B^{p'}_{s_1}(\mu^\prime_t)}$.   
Indeed, if $h\in H(\overline \D)$, then 
\begin{align*}
\|h\|_{B^1_{s_0+s_1-t}}
&\approx \sup_{0\ne b\in B^\infty_{-s_0-s_1}}
\frac{|\langle h,b\rangle_t|}{\|b\|_{B^\infty_{-s_0-s_1}}}
\approx \|h\|_{B^p_{s_0}(\mu_t)\odot B^{p'}_{s_1}(\mu^\prime_t)}.
\end{align*}
Since $h\in H(\overline \D)$ is dense in both spaces, we obtain the result.
\end{proof}

\subsection{Further remarks} 

Combining Theorem \ref{thm:predualCBps} with \eqref{eqn:correind} we can obtain 
characterizations of weak products of type $B^{p}_{s_0}(\mu_t)\odot B^{p'}_{s_1}(\mu_t)$ which generalize some of the results stated in Section 5 in \cite{Co-Ve}.

For instance, if $0<s<p$, then
\begin{align*}
\left(B^{p}_{0}(\mu_t)\odot B^{p'}_{-s}(\mu_t)\right)'
=\left(B^{p}_{s/p}(\mu_{t+s})\odot B^{p'}_{-s/p}(\mu_{t+s})\right)'\equiv CB^p_{s/p}(\mu_{t+s})=CB^p_0(\mu_t),
\end{align*} 
with respect to the pairing $\langle\cdot,\cdot\rangle_{t+s}$.

Observe that in the particular case $p=2$ and $t=0$, we have $CB^2_0=BMOA\equiv (H^1_{-s})'$, 
with respect to the pairing $\langle\cdot,\cdot\rangle_{s}$. 
Therefore, the above duality result and the fact that $B^2_s=H^2_s$ give 
$H^2\odot H^{2}_{-s}=H^1_{-s}$.

This unweighted weak factorization result can be generalized to the case $1<p<2$. In this case $B^p_0\subset H^p$, 
and we have that $CB^p_0=F^{\infty,p}_0$, where  $F^{\infty,p}_0$ denotes 
the Triebel-Lizorkin space 
of holomorphic functions on $\D$ such that the measure $d\mu_g(z)=|\p g(z)|^p(1-|z|^2)^{p-1}$ 
is a Carleson measure for $H^p$, that is $\mu_g(T_z)\lesssim (1-|z|^2)$ for any $z\in \D$ 
(see \cite{Ma-Sha}, p.178).
Since $F^{\infty,p}\equiv (F^{1,p'}_{-s})'$, 
with respect to the pairing $\langle\cdot,\cdot\rangle_{s}$, we have 
$
B^{p}_{0}\odot B^{p'}_{-s}=F^{1,p'}_{-s}.
$
Here, $F^{1,p'}_{-s}$ is the Triebel-Lizorkin space of holomorphic functions $g$ on $\D$ satisfying 
$$
\int_\bD \left(\int_{|1-\z \overline w|<1-|w|^2}| g(w)|^{p'}(1-|w|^2)^{sp'-2}d\nu(w)\right)^{1/p'}d\sigma(\z)<\infty.
$$

\end{document}